\documentclass[11pt,leqno]{amsart}
\usepackage{etex}
\usepackage{enumitem}
\usepackage[utf8x]{inputenc} 	      
\usepackage{lmodern}                
\usepackage[T1]{fontenc}            
\usepackage[french,english]{babel}  
\usepackage{ucs}             	      

\usepackage{graphicx}         
\usepackage{xcolor}           
\usepackage{enumitem}         
\usepackage{tikz}
\usetikzlibrary{fit,calc,positioning,decorations.pathreplacing,matrix, shapes}
\usetikzlibrary{trees}

\usepackage{amsmath, amsthm}  
\usepackage{amsfonts, amssymb}
\usepackage{stmaryrd}         
\usepackage{bm}               

\usepackage[ breaklinks,  
            pagebackref,           
            pdfborder={0 0 0}
            ]{hyperref}	 
\hypersetup{ pdfauthor={Bérénice Delcroix-Oger} }

\usepackage[nohints]{minitoc} 

 \let\footnote=\endnote
\ifx\thesection\undefined
\newtheorem{thm}{Theorem}
\else
\newtheorem{thm}{Theorem}[section]
\fi
 \theoremstyle{definition}
 \newtheorem{exple}[thm]{Example}

\theoremstyle{definition}
	\newtheorem{defi}{Definition}[section]

 \theoremstyle{plain}
 \newtheorem{prop}[thm]{Proposition}
 \newtheorem{lem}[thm]{Lemma}

 \newtheorem*{thmnonnum}{Theorem}

\theoremstyle{remark}
	\newtheorem{rque}{\textbf{Remark}}[section]	
	
\usepackage{xargs}

\newcommand{\comm}{\operatorname{Comm}}
\newcommand{\perm}{\operatorname{Perm}}
\newcommand{\lie}{\operatorname{Lie}}
\newcommand{\prelie}{\operatorname{PreLie}}
\newcommand{\F}{\mathcal{F}}
\newcommandx{\grd}[2][1=n,2=p]{\Pi_{#1,#2}}
\newcommandx{\hgrd}[2][1=n,2=p]{\widehat{\Pi_{#1,#2}}}
\newcommand{\dycmup}{\frac{\partial \ckp[-1]}{\partial y}}
\newcommand{\dxcmup}{\frac{\partial \ckp[-1]}{\partial x}}
\newcommand{\dxcmuc}{\frac{\partial \ckc[-1]}{\partial x}}
\newcommand{\p}{p} 
\newcommand{\np}{\ell} 

\newcommand{\eckp}[1][k]{\mathcal{C}_{#1}^{\bullet}}
\newcommand{\emup}{\mathcal{C}_{-1}^{\bullet}}
\newcommand{\ckp}[1][k]{\mathbf{C}_{#1}^{\bullet}}
\newcommand{\zkp}[1][k]{\mathbf{Z}_{#1}^{\bullet}}
\newcommand{\zmup}{\mathbf{Z}_{-1}^{\bullet}}
\newcommand{\eckc}[1][k]{\mathcal{C}_{#1}^{\times}}
\newcommand{\eckg}[1][k]{\mathcal{C}_{#1}^{g}}

\newcommand{\eckl}[1][k]{\mathcal{C}_{#1}^{l}}
\newcommand{\ckc}[1][k]{\mathbf{C}_{#1}^{\times}}
\newcommand{\zkc}[1][k]{\mathbf{Z}_{#1}^{\times}}
\newcommand{\ckg}[1][k]{\mathbf{C}_{#1}^{g}}
\newcommand{\zkg}[1][k]{\mathbf{Z}_{#1}^{g}}

\newcommand{\ckl}[1][k]{\mathbf{C}_{#1}^{l}}
\newcommand{\zkl}[1][k]{\mathbf{Z}_{#1}^{l}}
\newcommand{\m}{\mathfrak{m}}
\newcommandx{\pmax}[2][1=n,2=p]{\pi_{{#1},{#2}}}
\newcommandx{\grdp}[3][1=n,2=p, 3=1]{\Pi^{#3}_{#1,#2}}
\newcommandx{\pinko}[3][1=n,2=p, 3=o]{\pi_{#1,#2}^{#3}}
\newcommandx{\pinkmax}[3][1=n,2=p, 3=o]{M_{#1,#2}^{#3}}

\newcommand{\CP}{C(\mathcal{P})}
\newcommand{\HP}{H(\mathcal{P})}

\newcommandx{\cnlo}[3][1=\mathbf{n},2=\mathbf{p}, 3=\mathbf{o}]{c_{#1,#2}^{#3}}

\newcommandx{\Zx}{D_xf_1}
\newcommandx{\Zy}{D_yf_1}
\newcommand{\E}{\mathcal{E}}
\newcommand{\G}{\mathcal{G}}
\newcommand{\q}{\mathbf{q}}
\newcommand{\uplep}{\mathbf{p}}

\usepackage{endnotes}
\usepackage{stmaryrd}
\usepackage{textcomp}

\title{Semi-pointed partitions posets and species}

\author{B\'{e}r\'{e}nice Delcroix-Oger}

\thanks{ Institut Camille Jordan, UMR 5208, Universit\'{e} Claude Bernard Lyon 1 \\
 B\^{a}t. Jean Braconnier, 43 Bd du 11 novembre 1918, 69622 Villeurbanne Cedex \\
 \textbf{e-mail address:} oger@math. univ-lyon1. fr}
 \date{}

\begin{document}

\begin{abstract}
We define semi-pointed partition posets, which are a generalisation of partition posets and show that they are Cohen-Macaulay. We then use multichains to compute the dimension and the character for the action of the symmetric groups on their homology. We finally study the associated incidence Hopf algebra, which is similar to the Faà di Bruno Hopf algebra. 
\end{abstract}

\maketitle

\section*{Introduction}

The partition poset on a finite set $V$ is the well-known partially ordered set, or \emph{poset}, of partitions of $V$, endowed with the following partial order: a partition $P$ is smaller than another partition $Q$ if the parts of $P$ are unions of parts of $Q$. A variant of partition posets, called \emph{pointed partition posets}, has been studied by F. Chapoton and B. Vallette in \cite{ChVal} and \cite{Val}. A pointed partition of a set $V$ is a partition of $V$, with a distinguished element for each of its parts. The pointed partition poset on $V$ is then the set of pointed partitions of $V$, where a pointed partition $P$ is smaller than another pointed partition $Q$ if and only if the parts of $P$ are unions of parts of $Q$ and the set of pointed elements of $P$ is included in the set of pointed elements of $Q$.

A variant of partition posets and pointed partition posets naturally appears during the study of intervals in some hypertree posets (\cite{PhD}): we call them \emph{semi-pointed partition posets}. The link between the homology of posets of partitions and pointed partitions and operads $\lie$ and $\prelie$ raises naturally the question of the study of the homology of semi-pointed partition posets. Indeed, B. Fresse proved in \cite{Fressepart} that the homology of the partition posets was isomorphic to the Koszul dual of the operad $\comm$, which is the operad $\lie$, tensorised by the signature representation. Moreover, F. Chapoton and B. Vallette proved in \cite{ChVal} that the homology of the pointed partition posets was isomorphic to the Koszul dual of the operad $\perm$, which is the operad $\prelie$, tensorised by the signature representation. More generally, B. Vallette proved in \cite{Val} that the homology of a poset of partitions decorated by an operad is concentrated in higher degree if and only if the operad is Koszul and that, in this case, the homology of the poset is the Koszul dual of the operad, tensorised by the signature representation. The semi-pointed partition posets can then be seen as a poset of partition decorated by a coloured operad, which is a generalisation of classic operad, where the operad involved mixes both $\comm$ and $\perm$. This case raises the question of the extension of B. Vallette's results to coloured operad. We intend to study this question later.

We recall in the first section the generalities on coloured operads, 2-species and poset homology. We also explain how the study of the homology of some posets can be reduced to the study of multichains. After a short description of semi-pointed partition posets, we then show in the second section that the posets are Cohen-Macaulay by proving their total semi-modularity. In the third section, we use the theory of species to compute the dimension of the unique non trivial homology group of a given semi-pointed partition poset. This dimension is given by the following theorem:
\begin{thmnonnum} Let $V_1$ and $V_2$ be two finite sets of cardinality $\p$ and $\np$.
The sum of dimensions of the unique homology group of every maximal interval in the semi-pointed partition poset over $V_1$ and $V_2$, whose minimums have a unique part which is pointed, is given by:
\begin{equation*}
\frac{(\p+\np-1)!}{(\p-1)!} (\p+\np)^{\p-1}.
\end{equation*}

The dimension of the unique homology group of the augmented poset of semi-pointed partitions over $V_1$ and $V_2$ is given by:
\begin{equation*}
\frac{(\p+\np-1)!}{(\p-1)!} (\p+\np-1)^{\p-1}.
\end{equation*}

The sum of dimensions of the unique homology group of every  maximal interval in the semi-pointed partition poset over $V_1$ and $V_2$ is given by: 
\begin{equation*}
\frac{(\p+\np-1)!}{\p!} (\p+\np)^{\p}.
\end{equation*} 
\end{thmnonnum}
The reasoning used in this section relies on the link between the homology of the poset and multichains in the poset. We then compute the action of the symmetric groups on the homology of the semi-pointed partition poset, which can be compared to the known characters of both $\lie$ and $\prelie$ operads.

We then compute the incidence Hopf algebra associated to the hereditary family generated by maximal intervals in semi-pointed partition posets. We prove that it is isomorphic to the Hopf algebra structure on functions on formal diffeomorphisms in dimension $2$. It can therefore be seen as a generalisation of the Faà di Bruno Hopf algebra. This incidence Hopf algebra enables us to compute characteristic polynomials for maximal intervals whose greatest element is pointed. Some of the results appearing here were announced in the extended abstract \cite{fpsac15}.

 \section{Generalities}
 
 \subsection{Species and operads}
 \label{species} 

\subsubsection{Species and 2-species.}
 A \emph{species} is a functor from the category of finite sets and bijections $\operatorname{Bij}$ to itself. A \emph{2-species} is then a functor from $\operatorname{Bij} \times \operatorname{Bij}$ to $\operatorname{Bij}$. We refer the reader to \cite{BLL} for more details about these objects.
 
 \begin{exple} \label{exArs}  The map $\mathcal{A}_{r,s}$ which maps sets $V_1$ and $V_2$ with the set of forests of rooted trees whose roots are labelled by $V_1$ and whose other vertices are labelled by $V_2$ is a 2-species. It is represented on Figure \ref{Ars}.

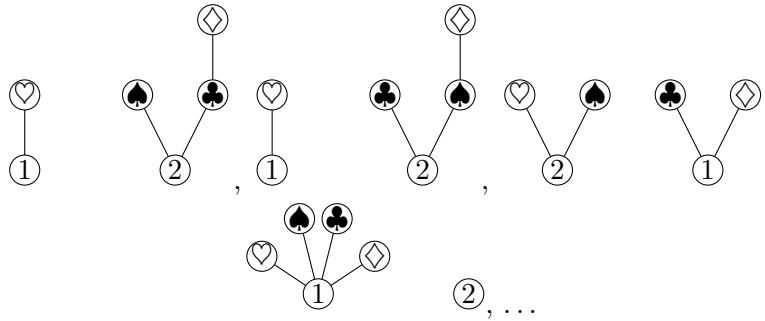
\begin{figure}[h!]
\begin{center}

\begin{tikzpicture}
\draw(0,0) -- (0,1);
\draw(1.5,1) -- (2,0) -- (2.5, 1) -- (2.5,2);
\draw[black, fill=white] (0,0) circle (0.2);
\draw[black, fill=white] (2,0) circle (0.2);
\draw[black, fill=white] (0,1) circle (0.2);
\draw[black, fill=white] (1.5,1) circle (0.2);
\draw[black, fill=white] (2.5,1) circle (0.2);
\draw[black, fill=white] (2.5,2) circle (0.2);
\draw(0,0) node{$1$};
\draw(0,1) node{$\heartsuit$};
\draw(1.5,1) node{$\spadesuit$};
\draw(2.5,1) node{$\clubsuit$};
\draw(2.5,2) node{$\diamondsuit$};
\draw(2,0) node{$2$};
\end{tikzpicture}, 
\begin{tikzpicture}
\draw(0,0) -- (0,1);
\draw(1.5,1) -- (2,0) -- (2.5, 1) -- (2.5,2);
\draw[black, fill=white] (0,0) circle (0.2);
\draw[black, fill=white] (2,0) circle (0.2);
\draw[black, fill=white] (0,1) circle (0.2);
\draw[black, fill=white] (1.5,1) circle (0.2);
\draw[black, fill=white] (2.5,1) circle (0.2);
\draw[black, fill=white] (2.5,2) circle (0.2);
\draw(0,0) node{$1$};
\draw(0,1) node{$\heartsuit$};
\draw(1.5,1) node{$\clubsuit$};
\draw(2.5,1) node{$\spadesuit$};
\draw(2.5,2) node{$\diamondsuit$};
\draw(2,0) node{$2$};
\end{tikzpicture}, 
\begin{tikzpicture}
\draw (0.5,1) -- (0,0) -- (-0.5,1);
\draw(1.5,1) -- (2,0) -- (2.5, 1) ;
\draw[black, fill=white] (0,0) circle (0.2);
\draw[black, fill=white] (0.5,1) circle (0.2);
\draw[black, fill=white] (-0.5,1) circle (0.2);
\draw[black, fill=white] (1.5,1) circle (0.2);
\draw[black, fill=white] (2.5,1) circle (0.2);
\draw[black, fill=white] (2, 0) circle (0.2);
\draw(0,0) node{$2$};
\draw(-0.5,1) node{$\heartsuit$};
\draw(0.5,1) node{$\spadesuit$};
\draw(1.5,1) node{$\clubsuit$};
\draw(2.5,1) node{$\diamondsuit$};
\draw(2,0) node{$1$};
\end{tikzpicture}, 
\begin{tikzpicture}
\draw (-0.75, 0.5) -- (0,0) -- (-0.25, 1);
\draw (0.75, 0.5) -- (0,0) -- (0.25, 1);
\draw[black, fill=white] (0,0) circle (0.2);
\draw[black, fill=white] (0.25,1) circle (0.2);
\draw[black, fill=white] (-0.25,1) circle (0.2);
\draw[black, fill=white] (0.75,0.5) circle (0.2);
\draw[black, fill=white] (-0.75, 0.5) circle (0.2);
\draw[black, fill=white] (2, 0) circle (0.2);
\draw(0,0) node{$1$};
\draw(-0.75,0.5) node{$\heartsuit$};
\draw(-0.25,1) node{$\spadesuit$};
\draw(0.25,1) node{$\clubsuit$};
\draw(0.75,0.5) node{$\diamondsuit$};
\draw(2,0) node{$2$};
\end{tikzpicture}, $\ldots$

\end{center}
\caption{Elements of the image by the species $\mathcal{A}_{r,s}$ of the sets $\{1,2\}$ and $\{\diamondsuit, \clubsuit, \heartsuit, \spadesuit\}$ \label{Ars}} 
\end{figure}

 The map which associates to two sets $V_1$ and $V_2$ the set of sets (resp. lists, pointed sets) on $V_1 \sqcup V_2$ is a 2-species.
\end{exple}

The following operations can be defined on 2-species.

\begin{defi}
Let $\F$ and $\G$ be two 2-species. The following operations are defined: 
\begin{itemize}
\item $(\F + \G )(I_1,I_2)=\F(I_1,I_2) \sqcup \G(I_1,I_2)$, (addition)
\item $(\F \cdot \G )(I,J)=\sum_{I_1 \sqcup I_2=I, J_1 \sqcup J_2=J} \F(I_1, J_1) \times \G(I_2, J_2)$, (product)
\item If $\E$ is a species and $\F(\emptyset,\emptyset)=\emptyset$,
\begin{equation*}
(\E \circ \F) (I_1, I_2)=\bigsqcup_{\pi \in \mathcal{P}(I_1 \cup I_2)} \E(\pi) \times \prod_{J \in \pi} \F(J_1, J_2),
\end{equation*} 
\end{itemize}
where $\mathcal{P}(I_1 \cup I_2)$ is the set of partitions of $I_1 \cup I_2$ and $J_i = J \cap I_i$.
(substitution)
We will sometimes omit $\circ$ and write $\E(\F)$ for $\E \circ \F$. 

\end{defi}
Especially, $\mathbb{E} \circ \F$ will denote the substitution of $\F$ in the species of sets $\mathbb{E}$, defined by $\mathbb{E}(V) = {V}$ for any set $V$, which satisfies: 
\begin{equation*}
(\mathbb{E} \circ \F) (I_1, I_2)=\bigsqcup_{\pi \in \mathcal{P}(I_1 \cup I_2)} \prod_{J \in \pi} \F(J_1, J_2).
\end{equation*}
To 2-species, the following generating series are also associated: 
\begin{defi} Let $\F$ be a 2-species.
The (exponential) generating series associated with $\F$ is defined by:
\begin{equation*}
\mathbf{C}_{\F}(x,y)=\sum_{n \geq 0}\sum_{m \geq 0} | \F(\{1,\ldots, n\}, \{1,\ldots, m\})| \frac{x^n y^m}{n! m!}. 
\end{equation*}

The cycle index series associated with $\F$ is the formal series in an infinite number of variables $\uplep=(p_1, p_2, p_3, \ldots ), \q=(q_1, q_2, q_3, \ldots )$ defined by:
\begin{equation*}
	\textbf{Z}_{\F}(\uplep, \q)= \sum_{\substack{n,m \geq 0 \\ (n,m) \neq (0,0)}} \frac{1}{n! m!} \left( \sum_{\sigma \in
	\mathfrak{S}_n, \tau \in \mathfrak{S}_m} \F^{\sigma, \tau} p_1^{\sigma_1} p_2^{\sigma_2} p_3^{\sigma_3}
	\ldots q_1^{\tau_1} q_2^{\tau_2} q_3^{\tau_3} \ldots \right),
	\end{equation*}
	where $\F^{\sigma, \tau}$  is the number of $\F$-structures on $(\{1,\ldots, n\}, \{1,\ldots, m\})$ fixed by the action of $\sigma$ on $\{1,\ldots, n\}$ and of $\tau$ on $\{1,\ldots, m\}$. 
\end{defi}

For $f=f(\uplep)$ and $g=g(\uplep, \q)$, to formal series in infinite variables $\uplep = (p_1, p_2, p_3, \ldots)$ and $\q = (q_1, q_2, q_3, \ldots)$, the \emph{plethystic substitution} $f \circ g$ is defined by: 
		\begin{align*}
		 k \mathit{th} &\text{ term} \\
		&\downarrow \\
		f \circ g (\uplep, \q) = f(\ldots, g(p_{k}, p_{2k}, p_{3k},& \ldots, q_{k}, q_{2k}, q_{3k}, \ldots), \ldots). 
		\end{align*}
		This substitution is linear and distributive on the left.

The suspension is defined as :
\begin{equation} \label{suspension}
\Sigma f(p_1, p_2, \ldots, q_1, q_2, \ldots) = - f(-p_1, -p_2, \ldots, -q_1, -q_2, \ldots).
\end{equation}		

These series satisfy the following relations: 

\begin{prop} Let $\F$ and $\G$ be two 2-species and $\E$ be a species. Their associated generating series satisfy: 

\begin{tabular}{cc}
$\mathbf{C}_{\F+\G} = \mathbf{C}_{\F} + \mathbf{C}_{\G} $,&
$\mathbf{C}_{\F \cdot \G} = \mathbf{C}_{\F} \times \mathbf{C}_{\G}$,\\
$\mathbf{Z}_{\F+\G} = \mathbf{Z}_{\F} + \mathbf{Z}_{\G} $,&
$\mathbf{Z}_{\F \cdot \G} = \mathbf{Z}_{\F} \times \mathbf{Z}_{\G}$,
\end{tabular}

and if $\G(\emptyset, \emptyset)=\emptyset $, 
\begin{tabular}{cc}
$\mathbf{C}_{\E \circ \G} = \mathbf{C}_{\E} \circ \mathbf{C}_{\G}$, & $\mathbf{Z}_{\E \circ \G} = \mathbf{Z}_{\E} \circ \mathbf{Z}_{\G}$.
\end{tabular}
\end{prop}

\begin{exple}
The map $\mathcal{A}_{r,s}$ in Example \ref{exArs} satisfies the following equation:
\begin{equation*}
\mathcal{A}_{r,s} = \mathbb{E} \circ \left(X_1 \times \mathbb{E}\left( \mathcal{A}_2\right) \right),
\end{equation*}
where the 2-species $X_1$ and $\mathcal{A}_2$ are defined as follows: $X_1(V_1, V_2) = V_1$ if $|V_1|=1$, $\emptyset$ otherwise and $\mathcal{A}_2(V_1,V_2)=\{$rooted trees on $V_2\}$.
\end{exple}

\subsubsection{Operads and coloured operads.}

We present in this paragraph the notion of operad and coloured operad. The reader can refer to \cite{LodVal} for more information about operads and to \cite{Pepijn} for more information about coloured operads. The name "operad" first appears in the article \cite{GILS} of P. May. This notion, introduced in the frame of iterated loop spaces, was then developed by other topologists such as M. Boardman and R. Vogt. In the context of homotopy theory, M. Boardman and R. Vogt especially introduced in the book \cite{BV73} the notion of coloured operad, which generalizes the notion of operad.


A (symmetric) operad $\mathcal{O}$ is a species with an additive structure: a map from $\mathcal{O} \circ \mathcal{O}$ to $\mathcal{O}$ satisfying some additional axioms. We recall the equivalent usual definition before introducing the definition for coloured operads: 
\begin{defi}
An \emph{operad} is a sequence of sets $\mathcal{P}(n)$ ($n \in \mathbb{N}$), with a right action of the symmetric group $\mathfrak{S}_n$ on $\mathcal{P}(n)$, and  with a distinguished element $\mathbf{1}$ in $\mathcal{P}(1)$ called the \emph{identity}, together with composition maps 
\begin{align*}
\circ: \mathcal{P}(n) \times \mathcal{P}(k_1) \times \ldots \times \mathcal{P}(k_n) \rightarrow \mathcal{P}(k_1  + \ldots + k_n) \\
(\alpha, \alpha_1, \ldots, \alpha_n) \mapsto \alpha \circ (\alpha_1, \ldots, \alpha_n),
\end{align*}
satisfying:
\begin{itemize}
\item associativity: 
\begin{multline*}
\alpha \circ \left(\alpha_1 \circ \left(\alpha_{1,1}, \ldots, \alpha_{1, k_1} \right), \ldots, \alpha_n \circ \left(\alpha_{n,1}, \ldots, \alpha_{n,k_n} \right) \right) \\= \left( \alpha \circ \left( \alpha_1, \ldots, \alpha_n \right) \right) \circ \left( \alpha_{1,1}, \ldots, \alpha_{1, k_1}, \ldots, \alpha_{n,1}, \ldots, \alpha_{n,k_n}\right)
\end{multline*}
\item identity: 
\begin{equation*}
\alpha \circ (\mathbf{1}, \ldots, \mathbf{1}) = \alpha = \mathbf{1} \circ \alpha
\end{equation*}
\item equivariance: the composition commutes with the action of the symmetric group.
\end{itemize}
\end{defi}

\begin{exple} \label{esp}
\begin{itemize}
\item Operad $\comm$ is the operad which underlying species is the species of non empty sets $\mathbb{E}-1$. The composition is given by considering a set of sets as a set.
\item Operad $\perm$ is the operad which underlying species is the species of pointed sets and which associates to a set an element of it (its pointed element). The composition is given by pointing the pointed element of the pointed set. For instance, $\{1, \mathbf{2}, 3\}$ $\circ$ $(\{\mathbf{1}\},$ $\{2,3, \mathbf{4}\},$ $\{5,\mathbf{6}\})$ is $\{1,2,3,\mathbf{4},5,6\}$.
\item Operad $\prelie$ is the operad which underlying species is the species of rooted trees, introduced in \cite{ChLiv}.
\item Operad $\lie$ is the Koszul dual of operad $\comm$. The suspension of its cycle index series is the inverse for the substitution of the cycle index series of $\comm$.
\end{itemize}
\end{exple}

The notion of 2-coloured operads is a special case of the notion of coloured operads. It is given by the following definition: 
\begin{defi}
A \emph{2-coloured operad} is a sequence of sets $\mathcal{P}(c_1, \ldots, c_n;c)$ ($n  \in \mathbb{N} ; c_1, \ldots, c_n, c \in \{0,1\}$), with a right action of the symmetric group $\mathfrak{S}_n$,  with a distinguished element $\mathbf{1}_0 \in \mathcal{P}(0;0)$ (resp. $\mathbf{1}_1 \in \mathcal{P}(1;1)$) called the \emph{identity} on $0$ (resp. $1$), together with a composition map 
\begin{multline*}
\circ: \mathcal{P}(c_1, \ldots, c_n;c) \times \mathcal{P}(d_{1,1}, \ldots, d_{1,k_1};c_1) \times \ldots 
\\ \times \mathcal{P}(d_{n,1}, \ldots, d_{n,k_n};c_n) \rightarrow \mathcal{P}(d_{1,1}, \ldots, d_{n,k_n};c) 
\end{multline*}
satisfying associativity, identity and equivariance axioms. The diagrams involved are fully depicted in \cite[§2]{EM06}.
\end{defi}
 
 The underlying sets, together with the action of the symmetric groups $\mathfrak{S}_k \times \mathfrak{S}_{n-k}$, permuting elements of $V_1 \sqcup V_2$, but fixing $V_1$ and $V_2$, can be seen as a 2-species.

	\subsection{Poset homology}
	
		\label{truc}
	The reader can refer to \cite{Wachs} for an introduction on poset topology. We briefly recall here the notions. Let $P$ be a finite poset.

\begin{defi}
A \emph{strict $m$-chain} in the poset $P$ is a $m$-tuple $(a_1,\ldots, a_m)$ where $a_i$ are elements of $P$, neither maximum nor minimum in $P$, and $a_i \prec a_{i+1}$, for all $i\geq 1$. We write $\mathcal{C}_m(P)$ for the set of strict $(m+1)$-chains and $C_m(P)$ for the $\mathbb{C}$-vector space generated by all strict $(m+1)$-chains.
\end{defi}

The set $\cup_{m \geq 0}{\mathcal{C}_{m}(P)}$ defines then a simplicial complex. The homology of $P$ is defined as the homology of this simplicial complex.

Define the linear map $d_m:C_{m+1}(P) \rightarrow C_{m}(P)$ which maps a $(m+1)$-simplex to its boundary by :
\begin{equation*}
d_m(a_0<\ldots <a_m)=\sum_{i=0}^m (-1)^{i-1} a_0<\ldots<a_{i-1}<a_{i+1}<\ldots < a_m.
\end{equation*} 
These maps satisfy $d_{m-1} \circ d_m=0$. The obtained pairs $(C_m(P),d_m)_{m>0}$ form a \emph{chain complex}. Thus, we can define the homology of the poset.

\begin{defi}The homology group of dimension $m$ of the poset $P$ is: 
\begin{equation*}
H_m(P)=\operatorname{Ker} d_m /\operatorname{Im} d_{m+1}.
\end{equation*}
\end{defi}

We consider in this article the reduced homology, written $\tilde{H}_i$. This homology is obtained by adding a vector space $C_{-1}(P)=\mathbb{C}. e$, and the trivial linear map $d: C_0 \rightarrow C_{-1}$,which maps every singleton to the element $e$. The reduced homology differs from the homology only on the first homology group $\tilde{H_0}$. The first homology group then satisfies:
\begin{equation*} \operatorname{dim} (\tilde{H_0}(P))=\operatorname{dim} (H_0(P))-1.
\end{equation*}

Dimensions of the homology groups satisfy the following well-known property: 
\begin{lem}[\cite{Wachs}, Euler-Poincaré formula] The Euler characteristic $\chi$ (or equivalently Möbius numbers) of the homology satisfies: 
\begin{equation}\label{CarEuler1}
\chi =\sum_{m \geq 0} (-1)^{m} \operatorname{dim} \tilde{H}_m(P) = \sum_{m \geq -1} (-1)^{m} \operatorname{dim} C_m(P).
\end{equation}
\end{lem}

This relation can be generalized to some characters. Let $G$ be a finite group and $(P, \preceq)$ be a $G$-poset, i.e. a poset together with a $G$-action on its elements that preserves the partial order. The poset $P$ is not necessarily bounded but will be required to be finite and pure, i.e. all its maximal strict chains have the same finite cardinality $\dim P$. Then the complex reduced homology groups $\widetilde{H}_j(\mathbb{C})$ are $G$-modules (see \cite[2.3]{Wachs} for details). We now explain how characters for the action of $G$ on the homology groups can be obtained thanks to the character for the action of $G$ on $k$-multichains. We will denote by $\chi_{\widetilde{H}_{i}}$ the character for the action of $G$ on the homology group $\widetilde{H}_{i}$ of $P$, $\chi^s_i$ the character for the action of $G$ on the vector space $C_i$ spanned by strict $i+1$-chains and by $\chi^l_k$ the character for the action of $G$ on the vector space generated by $k$-multichains. This link relies on the Hopf trace formula \cite[Theorem 2.3.9]{Wachs}:

\begin{equation}
\mu_\chi:= \sum_{i=0}^{\dim P}(-1)^{i} \chi_{\widetilde{H}_{i}} =  \sum_{i=-1}^{\dim P} (-1)^{i} \chi^s_i
\end{equation}

This formula is the heart of the reasoning below : the study of homology groups will be obtained through the study of chains. We introduce the notion of multichains:

\begin{defi}
A \emph{$k$-multichain} in the poset $P$ is a $k$-tuple $(a_1, \ldots, a_k)$ of elements of $P$ such that $a_i \preceq a_{i+1}$ for all $1 \leq i \leq k-1$.
\end{defi}

The number of $n$-multichains is a polynomial in $n$ called the \emph{zeta polynomial}. This polynomial was introduced by R. Stanley in \cite{Stanzet} and further developed by P. Edelman in \cite{Edelzet}. The link between zeta polynomials and Möbius numbers can be found in \cite[Prop 3.12.1]{Stan1} for bounded posets and posets without any least or greatest element. We extend it here to posets with a least or a greatest element and write it in terms of group representations. This reasoning was already used in \cite{mar1}, where it enabled the author to compute the action of the symmetric group on the unique homology group of the hypertree poset. 

 We now link multichains and strict chains. 

\begin{prop} The character for the action of $G$ on the vector space generated by $k$-multichains of the pure finite poset $P$, $k \geq 1$, is given by:
\begin{itemize}
\item if $P$ is bounded,
\begin{equation} \label{3}
\chi^l_k = \sum_{i = -1}^{\dim P} \binom{k+1}{i+2} \chi^s_i,
\end{equation}
\item if $P$ is not bounded but has a least or a greatest element,
\begin{equation}\label{4}
\chi^l_k = \sum_{i =-1}^{\dim P} \binom{k}{i+1} \chi^s_i,
\end{equation}
\item otherwise,
\begin{equation}\label{5}
\chi^l_k = \sum_{i =0}^{\dim P} \binom{k-1}{i} \chi^s_i.
\end{equation}
\end{itemize}
\end{prop}

\begin{proof}
The principle of this proof is to establish a one-to-one correspondence between $k$-multichains, and a pair $(w,sc)$, where $w$ is a word on the alphabet $\{0,1\}$ and $sc$ is a strict chain in the poset. We denote by $\hat{0}$ the least element and $\hat{1}$ the greatest one, if they exist.

Let us consider a $k$-multichain $(a_1,\ldots, a_k)$ in the poset $P$. When deleting repetitions and extrema in the chain, we obtain a strict chain $(a_{i_1}, \ldots, a_{i_s})$. Moreover, we encode the multichain with the word $(w_1, \dots, w_k)$, where $w_i \in \{0,1\}$ for all $i$ and
\begin{itemize} 
\item $w_1=0$ if $a_1=\hat{0}$, $1$ otherwise;
\item $w_{j+1}=0$ if $a_j = a_{j+1}$, $1$ otherwise, if $1 \leq j \leq k-1$.
\end{itemize}

\begin{description}
\item[Case 1] If the poset is bounded, then the number of $1$ in $w$ corresponds to the number of elements in the strict chain $sc$, if the multichain does not contain the maximum, or is one more than the number of element in the strict chain $sc$ otherwise. Moreover, if the multichain only contains extrema, $sc$ is empty.

\item[Case 2] If the poset contains exactly one extremum, then the number of $1$ in $w$ corresponds to the number of element in the strict chain $sc$. Moreover, if the multichain only contains $k$ times the extremum, $sc$ is empty.

\item[Case 3] If the poset contains no extremum, then $w_1=1$. The number of $1$ in $w$ corresponds to the number of elements in the strict chain $sc$ and the first letter of $w$ is a $1$. If $sc$ is a strict $i$-chain, there are $\binom{k-1}{i}$ different possible words. Moreover, as $k \geq 1$, there is at least one element in $sc$.
\end{description}

As $P$ is a $G$-poset, the partial order of $P$ is preserved by the action of $G$. Then, if an element $g$ of $G$ sends the $k$-multichain $(a_1,\ldots, a_k)$ to the $k$-multichain $(b_1,\ldots, b_k)$, we have $b_i=g  \cdot a_i$ for all $i$ and $g$ will send the associated strict chain $(a_{i_1}, \ldots, a_{i_s})$ to the chain $(b_{i_1}, \ldots, b_{i_s})$. As the elements $a_{i_j}$ are pairwise different, so are the elements $b_{i_j}$. Moreover, as the action of $G$ preserves the partial order, the extrema are fixed points for this action, so as none of the $a_{i_j}$ are extrema, none of the $b_{i_j}$ are extrema and the chain $(b_{i_1}, \ldots, b_{i_s})$ is strict.
The bijection defined above is thus compatible with the $G$-action.
\end{proof}

The right part of Expressions \eqref{3}, \eqref{4} and \eqref{5} are polynomial in $k$: they are then well-defined for non-positive integers $k$. Using the value $\binom{-1}{i} = (-1)^i$, we obtain immediately:

\begin{prop} The alternate sum of characters for the action of $G$ on the homology of the poset $P$ is given by:
\begin{itemize}
\item if $P$ is bounded,
\begin{equation}
\mu_\chi =  \chi^l_{-2}
\end{equation}
\item if $P$ is not bounded but has a least or a greatest element,
\begin{equation}
\mu_\chi =  - \chi^l_{-1} 
\end{equation}
\item otherwise,
\begin{equation}
\mu_\chi = \chi^l_{0} -1.
\end{equation}
\end{itemize}
\end{prop}

We apply this proposition in a particular case: when the poset is Cohen-Macaulay.

\begin{defi}
A finite pure poset $P$ is \emph{Cohen-Macaulay} if all its homology groups but the one in highest degree $\dim P$ vanish. 
\end{defi}

\begin{exple} The poset of subsets of a finite set $I$, ordered by inclusion, is Cohen-Macaulay.
The partition posets and pointed partition posets are Cohen-Macaulay.
\end{exple}

When the poset $P$ is Cohen-Macaulay, the character for the action of $G$ on its unique homology group is given by the alternating sum of characters for the action of $G$ on vector spaces of strict chains, according to the Hopf trace formula \cite[Theorem 2.3.9]{Wachs}. We then obtain:

\begin{prop} \label{-1}
The character for the action of $G$ on the unique homology group of the poset $P$ is given by:
\begin{itemize}
\item if $P$ is bounded,
\begin{equation}
\chi_{\widetilde{H}_{\dim P}} = (-1)^{\dim P}  \chi^l_{-2}
\end{equation}
\item if $P$ is not bounded but has a least or a greatest element,
\begin{equation}
\chi_{\widetilde{H}_{\dim P}}:=  (-1)^{\dim P-1} \chi^l_{-1} 
\end{equation}
\item otherwise,
\begin{equation}
\chi_{\widetilde{H}_{\dim P}} = (-1)^{\dim P} (\chi^l_{0} -1).
\end{equation}
\end{itemize}
\end{prop}

\begin{exple}
Let us illustrate the formulae in the three following cases.
\begin{figure}[h!]
\begin{center}
\begin{tikzpicture}
\draw[black] (0, 0) -- (1,1) -- (1,0);
\draw[black] (1,0) -- (0,1) -- (0,0);
\draw[purple] (0,0) -- (0.5,-0.5) -- (1,0);
\draw[purple] (0,1) -- (0.5,1.5) -- (1,1);
\draw[black, fill=white] (0,0) circle (0.2);
\draw[black, fill=white] (1,1) circle (0.2);
\draw[black, fill=white] (1,0) circle (0.2);
\draw[black, fill=white] (0,1) circle (0.2);
\draw[purple, fill=white] (0.5,-0.5) circle (0.2);
\draw[purple, fill=white] (0.5,1.5) circle (0.2);
\draw (0.5,1.5) node{M};
\draw (0.5,-0.5) node{m};
\draw (0,0) node{a};
\draw (1,0) node{b};
\draw (0,1) node{c};
\draw (1,1) node{d};
\draw (0.5,-1) node{(1)};
\end{tikzpicture} \hspace{1cm}
\begin{tikzpicture}
\draw[black] (0, 0) -- (1,1) -- (1,0);
\draw[black] (1,0) -- (0,1) -- (0,0);
\draw[purple] (0,0) -- (0.5,-0.5) -- (1,0);
\draw[black, fill=white] (0,0) circle (0.2);
\draw[black, fill=white] (1,1) circle (0.2);
\draw[black, fill=white] (1,0) circle (0.2);
\draw[black, fill=white] (0,1) circle (0.2);
\draw[purple, fill=white] (0.5,-0.5) circle (0.2);
\draw (0.5,-0.5) node{m};
\draw (0,0) node{a};
\draw (1,0) node{b};
\draw (0,1) node{c};
\draw (1,1) node{d};
\draw (0.5,-1) node{(2)};
\end{tikzpicture}\hspace{1cm}
\begin{tikzpicture}
\draw[black] (0, 0) -- (1,1) -- (1,0);
\draw[black] (1,0) -- (0,1) -- (0,0);
\draw[black, fill=white] (0,0) circle (0.2);
\draw[black, fill=white] (1,1) circle (0.2);
\draw[black, fill=white] (1,0) circle (0.2);
\draw[black, fill=white] (0,1) circle (0.2);
\draw (0,0) node{a};
\draw (1,0) node{b};
\draw (0,1) node{c};
\draw (1,1) node{d};
\draw (0.5,-1) node{(3)};
\end{tikzpicture}
\end{center}
\end{figure}

 The Möbius number of these posets is $\mu = -1$. The number of strict chains satisfy $\chi^s_0 = 1$, $\chi^s_1 = 4$ and $\chi^s_0 = 2$. We then have: 
\begin{enumerate}
\item $\chi^l_k = 4 + 4(k-1)$, and the equality $\mu = -1 + 4 + 4(-1)$ is satisfied
\item $\chi^l_k = 1 + 4k + 2k(k-1)$, and the equality $\mu = -\left( 1 - 4 + 2 \times 2\right)$ is satisfied
\item $\chi^l_k = k+1 + 2k(k+1) + 4\frac{(k+1)k(k-1)}{6}$, and the equality $\mu = -2+1$ $+ 2(-2)(-1)$ $+ 4\frac{(-1)(-2)(-3)}{6}$ is satisfied
\end{enumerate}
\end{exple}

\section{Semi-pointed partitions posets: presentation and Cohen-Macaulayness}

Let us first define semi-pointed partition posets and show that their homology is concentrated in maximal degree.

\subsection{Semi-pointed partitions posets}

Let $V_1$ and $V_2$ be two finite sets.
The semi-pointed partition poset on $V_1$ and $V_2$ can be viewed as a poset of partitions on $V_1 \sqcup V_2$ decorated by a 2-coloured operad, which is a generalisation of the partition posets decorated by operads described in the article \cite{Val} of B. Vallette.

Let us first describe the involved 2-coloured operad:

\begin{defi}
The $PSP$ operad is the 2-coloured operad defined by:
\begin{itemize}
\item $PSP(0, \ldots, 0 ; 1)$ and $PSP(1, \ldots, 1;0)$ are empty, $PSP(c_1, \ldots, c_n; 1)$ is the set of pointed sets $\{1, \ldots, \mathbf{i}, \ldots, n\}$, pointed in a $i$ satisfying $c_i=1$,  and $PSP(c_1, \ldots, c_n; 0)$ is the set $\{\{1, \ldots, n\}\}$. 
\item The composition of a (possibly pointed) set $E_1$ in the element $x$ of the (possibly pointed) set $E_2$ is:
\begin{description}
\item[if $E_2$ is pointed in $x$] the set $E_1 \cup E_2-\{x\}$ pointed in the pointed element of $E_1$ if there is one, or is not pointed otherwise,
\item[if $E_2$ is pointed in an element $y$ and $x \neq y$] the set $E_1 \cup E_2-\{x\}$ pointed in $y$
\item[otherwise] the non-pointed set $E_1 \cup E_2-\{x\}$.
\end{description}
\end{itemize}
We keep the same notation $PSP$ for the operad and the underlying species.
\end{defi}

The semi-pointed partitions admit the two following equivalent definitions. 

\begin{defi}
Let $V_1$ and $V_2$ be two sets of cardinality $\p$ and $\ell$. A \emph{semi-pointed partition} of $(V_1, V_2)$ is a partition of $V_1 \sqcup V_2 = (P_1, \dots, P_k)$ such that to each part $P_j$ is associated an element of $PSP(V_1 \cap P_j, V_2 \cap P_j)$.

An alternative definition is that  a \emph{semi-pointed partition} is a partition of $V_1 \sqcup V_2$ such that each part in the partition satisfies:
\begin{itemize}
\item If all elements in the part belong to $V_1$, the part is pointed in one of its element,
\item If all elements in the part belong to $V_2$, the part is not pointed,
\item If some elements in the part belong to $V_1$ and other to $V_2$, the part can be not pointed or pointed in one of its elements belonging to $V_1$.
\end{itemize} 
A $(n,p)$-semi-pointed partition is a semi-pointed partition of $V=\{1, \ldots, n\}$ with $V_1 = \{1, \ldots, p\}$ and $V_2= \{p+1, \ldots, n\}$.
\end{defi}
We will write in bold the pointed element in each part.

\begin{exple}
The set of $(3,2)$-semi-pointed partitions is the following:
$\{\mathbf{1}\}\{\mathbf{2}\}\{3\}$, $\{\mathbf{1}\}\{\mathbf{2},3\}$,$\{\mathbf{1}\}\{2,3\}$, $\{\mathbf{2}\}\{\mathbf{1},3\}$, $\{\mathbf{2}\}\{1,3\}$, $\{3\}\{\mathbf{1},2\}$, $\{3\}\{1,\mathbf{2}\}$, $\{\mathbf{1},2,3\}$, $\{1,\mathbf{2}, 3\}$, $\{1,2,3\}$.
\end{exple}

Let $V$ be a finite set. The set of semi-pointed partitions on $V=V_1 \sqcup V_2$ can be endowed with the following partial order:

\begin{defi}  Let $P$ and $Q$ be two semi-pointed partitions. The partition $P$ is smaller than the partition $Q$ if and only if the parts of $P$ are unions of parts of $Q$ and the pointing of parts of $P$ is "inherited" from the ones of $Q$, meaning that if a part $p$ of $P$ is union of parts $(q_1, \ldots, q_k)$ of $Q$, then the pointing of $p$ is chosen in those of the $q_i$, given that if one of the $q_i$ is not pointed, the part $p$ can be not pointed. We will say that we \emph{merge} the parts $(q_1, \ldots, q_n)$ into $p$.
\end{defi}

\begin{exple} With $V_1 = \{1, 2, 3\}$ and $V_2 = \{4,5\}$, the semi-pointed partition  $\{\mathbf{1}, 2\}\{3,4\}\{5\}$ is smaller than the semi-pointed partition $\{\mathbf{1},2,3,4\}\{5\}$ but cannot be compare to the semi-pointed partition $\{1,2,\mathbf{3}, 4\}\{5\}$. 
\end{exple}

\begin{rque}
When the semi-pointed partitions are viewed as partitions decorated by the coloured operad $PSP$, the definition of the associated partial order can be seen as a generalisation of the order defined on partitions decorated by an operad to coloured operad: if $P$ and $Q$ are two semi-pointed partitions, with $P$ smaller than $Q$, and if a part $p$ of $P$ is union of parts $(q_1, \ldots, q_n)$ of $Q$, the element of $PSP(p \cap V_1, p \cap V_2)$ chosen as a decoration of the part $p$ is chosen in the composition of the decorations.
\end{rque}

We denote by $\widehat{\grd[V][V_1]}$ the poset of semi-pointed partitions of $V= V_1 \sqcup V_2$ bounded by the addition of a least element $\widehat{0}$ and $\widehat{\grd}$ the poset of $(n,\p)$-semi-pointed partitions bounded by the addition of a smallest element $\widehat{0}$. The maximal intervals in $\grd$ whose least element is pointed are all isomorphic: we write $\grdp$ the maximal interval in $\grd$ whose least element is pointed in $1$. We also write $\grdp[n][\p][0]$ for the maximal interval in $\grd$ whose least element is not pointed. The partition whose parts are all of size $1$, endowed with the unique way to point the parts, will be denoted by $\pmax[V][V_1]$. It is the greatest element of $\grd[V][V_1]$.

Let us remark that two semi-pointed partitions $P$ and $Q$ can be in several different posets  $\grd[V][V_1]$. If the partition $P$ is inferior to the partition $Q$ in one of these posets, then it is the case in all other posets containing these two elements: the order between these elements does not depend on the underlying sets $V_1$ and $V_2$.

\begin{rque}
The poset $\grdp[n][0][0]$ is the poset of partitions of $\{1, \ldots n\}$. 

The posets $\grd[n][n]$ and $\grd[n][n-1]$ are two posets isomorphic to the pointed partition poset. This isomorphism comes by definitions for $\grd[n][n]$ and by identifying non pointed parts with parts pointed in the last element $n$ for $\grd[n][n-1]$.
\end{rque}

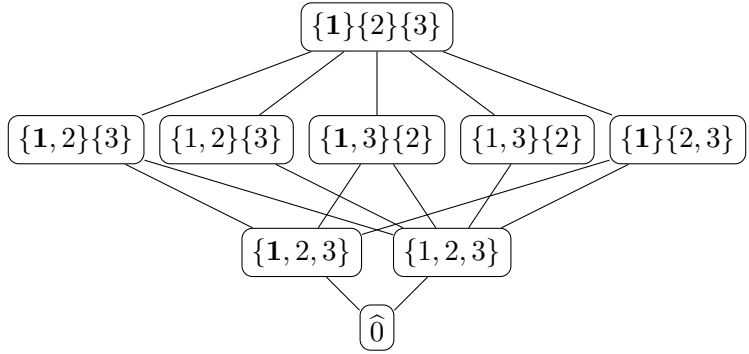
\begin{figure}[h]
\begin{center}
\begin{tikzpicture}
\node[draw, rounded corners] (max) at (0,-0.5) {$\{\mathbf{1}\}\{2\}\{3\}$};
\node[draw, rounded corners] (1) at (-1,-3.5) {$\{\mathbf{1},2,3\}$};
\node[draw, rounded corners] (0) at (1,-3.5) {$\{1,2,3\}$};
\node[draw, rounded corners] (13) at (-4,-2) {$\{\mathbf{1}, 2\}\{3\}$};
\node[draw, rounded corners] (00a) at (-2,-2) {$\{1, 2\}\{3\}$};
\node[draw, rounded corners] (12) at (0,-2) {$\{\mathbf{1}, 3\}\{2\}$};
\node[draw, rounded corners] (00b) at (2,-2) {$\{1, 3\}\{2\}$};
\node[draw, rounded corners] (10) at (4,-2) {$\{\mathbf{1}\}\{2,3\}$};
\node[draw, rounded corners] (min) at (0,-4.5) {$\widehat{0}$};
\draw (0) -- (min) -- (1);
\draw (0) -- (13);
\draw (0) -- (00a);
\draw (0) -- (00b);
\draw (0) -- (12);
\draw (0) -- (10);
\draw (1) -- (13);
\draw (1) -- (12);
\draw (1) -- (10);
\draw (max) -- (13);
\draw (max) -- (00a);
\draw (max) -- (00b);
\draw (max) -- (12);
\draw (max) -- (10);
\end{tikzpicture}
\caption{Poset $\widehat{\grd[3][1]}$. \label{pos3}}
\end{center}
\end{figure}

\subsection{Cohen-Macaulayness}

A poset $P$ is \emph{totally semi-modular} if for any interval $I$ in $P$, and for any elements $x,y$ in $I$ which cover an element $z$ in $I$, then there exists an element $t$ in $I$ covering $x$ and $y$. 
It follows from \cite{BjW} that any bounded, graded totally semi-modular poset is Cohen-Macaulay, i.e. has its homology concentrated in the highest degree. We then prove the following proposition by total semi-modularity:

\begin{prop}\label{pitsm}
The duals of the posets $\hgrd$, $\grdp[n][\p][0]$ and $\grdp[n][\p][1]$ are totally semi-modular and thus the posets $\hgrd$, $\grdp[n][\p][0]$ and $\grdp[n][\p][1]$ are Cohen-Macaulay (by \cite[Corollary 5.2]{BjW}).
\end{prop}

\begin{proof} 

Let $[a,b]$ be an interval of the dual $\hgrd^*$ of $\hgrd$. The greatest element of $\hgrd^*$ is denoted by $\widehat{1}$. Let us suppose that there exist three elements  $t,x,y$ in $[a,b]$ such that $x$ and $y$ cover $t$ and such that $x$ is different from $y$. We have to show the existence of an element $z$ in $[a,b]$ covering $x$ and $y$.

If $x$ and $y$ have only one part, $b=\widehat{1}$ and the elements $x$ and $y$ are covered by $b$, so we can choose $z=b$.

Otherwise, we can consider that $x$ is obtained from $t$ by merging the parts $p_1$ and $p_2$ of $t$ with a pointed element $e_x$, if the part is pointed. The element $e_x$, if it exists, can be the pointed element of $p_1$, if it exists, or the pointed element of $p_2$, if it exists. We can also consider that $y$ can be obtained by merging the parts $p_3$ and $p_4$ of $t$ with a pointed element $e_y$, if the part is pointed. The element $e_x$, if it exists, can be the pointed element of $p_3$, if it exists, or the pointed element of $p_4$, if it exists. In the following, "choosing the element $e_x$ for a part $p$", will mean that either $e_x$ exists and $e_x$ is pointed in $p$ or that $p$ is not pointed. 

\begin{itemize}
\item \underline{First case}: The parts $p_1$, $p_2$, $p_3$ and $p_4$ are pairwise different.

Then, we consider the element $z$ obtained from $t$ by merging $p_1$ and $p_2$ and choosing the pointed element $e_x$ for the union of the parts, and by merging $p_3$ and $p_4$ and choosing the pointed element $e_y$ for the union of the parts. This element $z$ can be obtained from $x$ by merging $p_3$ and $p_4$ and choosing the element  $e_y$ for the union of parts and from $y$ by merging $p_1$ and $p_2$ and choosing the element  $e_x$ for the union of parts: $z$ thus covers both $x$ and $y$ and belongs to the interval $[a,b]$.

\item \underline{Second case}: The parts $p_1$, $p_2$, $p_3$ and $p_4$ are not pairwise different. 

Then, up to a renumbering of parts, we can consider that $p_1$ and $p_4$ are equal. Note that there cannot be more than two equal parts as $x$ and $y$ are different and parts $p_{2i-1}$ and $p_{2i}$ are also different, for $i \in \{1, 2\}$. Let us now find an element covering both $x$ and $y$. The underlying partition of such an element will be obtained from $t$ by merging $p_1$, $p_2$ and $p_3$ in one part. We now have to choose a compatible pointed element for this new part in order to cover $x$ and $y$ and to be inferior to $b$.

\begin{description}
\item[On the one hand] If $b\neq \widehat{1}$ and the pointed element of one part in $b$ comes from a pointed element in a part $p_i$ of $p$, as $b$ is greater than $x$ and $y$, the three parts $p_1$, $p_2$ and $p_3$ are in only one part of $b$, pointed in $e_p$. We then choose $z$ to be the partition obtained from $t$ by merging the parts $p_1$, $p_2$ and $p_3$ and by choosing the pointed element $e_p$. The partition $z$ covers $x$ and $y$ and is smaller than $b$. Indeed, $e_p$ can be chosen as the pointed element for the union of $p_1 \cup p_2$ and $p_3$ in $x$ and for the union of $p_2$ and $p_1 \cup p_3$ in $y$. Moreover, the parts of $b$ can be obtained as union of parts of $z$.

\item[On the other hand] Otherwise, we can choose the element $z$ obtained from $t$ by merging the three parts $p_1$, $p_2$ and $p_3$ and choosing their pointed element as follows: 
\begin{itemize} 
\item If $x$ and $y$ both have inherited of the pointed element in $p_1$, $z$ inherits of it.
\item Otherwise, if $x$ or $y$ has inherited of the pointed element in $p_i$, for $i \neq 1$, $z$ inherits of the pointed element in $p_i$.
\end{itemize}
\end{description}

\end{itemize} 

\end{proof}

We study the unique non trivial homology group associated to semi-pointed partition posets in the following section.

\section{Homology of the semi-pointed partition poset}

We now apply the result \ref{-1} of section \ref{truc} to semi-pointed partition posets. In the unbounded poset of semi-pointed partitions on $(V_1, V_2)$ $\grd[V][V_1]$, a strict $k$-chain (resp. $k$-multichains) \emph{with multiplicity} is a strict $k$-chain (resp. $k$-multichains) in one of the maximal interval of the poset. It is equivalent to the data of $(Ch_k, \min)$, where $\min$ is a partition in only one part of the poset and $Ch_k$ is a strict chain (resp. multichain) of the poset whose elements are greater than $\min$.

For $k \geq 1$, we denote by: 
\begin{itemize}
\item $\eckg$ the species of $k$-multichains in semi-pointed partitions posets
\item $\eckl$ the species of $k$-multichains with multiplicity in semi-pointed partitions posets
\end{itemize}

The generating series and cycle index series associated to these species are respectively given by $\ckg$, $\ckl$ and $\zkg$ and $\zkl$.

Hence, applying Proposition \ref{-1}, we obtain:

\begin{prop}\label{mu1} The action of $\mathfrak{S}_\p \times \mathfrak{S}_{n-\p}$ induced on the unique homology group of the poset $\hgrd$ is given by the opposite of the value in $k=-1$ of the polynomial $\zkg$. The action of $\mathfrak{S}_\p \times \mathfrak{S}_{n-\p}$ induced on the direct sum of the homology of maximal intervals in the poset $\grd$ is given by the value in $k=-2$ of the polynomial $\zkl$.
\end{prop}

The aim of this section is then the computation of the dimension and the action of the symmetric groups on the unique non trivial homology group of $\hgrd$. 

\subsection{Relations between species}
Let $k$ be a positive integer.

We will need the following auxiliary species: 
\begin{itemize}
\item $\eckp$, the species which associates to $(V_1, V_2)$ the set of $k$-multichains in $\grd[V_1 \cup V_2][V_1]$, whose minimum is a partition with only one part, which is pointed, (called $k$-pm-multichains)
\item $\eckc$, the species which associates to $(V_1, V_2)$ the set of $k$-multichains in $\grd[V_1 \cup V_2][V_1]$, whose minimum is a partition with only one part, which is not pointed (called $k$-um-multichains).
\end{itemize}
The generating series and cycle index series associated to these species are $\ckp$, $\ckc$ and $\zkp$ and $\zkc$.

The species are linked by the following relations. 

\begin{prop}\label{relesp} The species $\eckp$, $\eckc$, $\eckg$ and $\eckl$ are linked by:
\begin{equation*}
\eckp=\eckp[k-1] \times \mathbb{E} \circ \left(\eckp[k-1] + \eckc[k-1] \right),
\end{equation*}
\begin{equation*}
\eckc= \left( \mathbb{E}-1\right) \circ \ \eckc[k-1] \times \mathbb{E} \circ {\eckp[k-1]} , 
\end{equation*}
\begin{equation*}
\eckg= \left( \mathbb{E}-1\right) \circ \left( \eckp[k] + \eckc[k]\right) ,
\end{equation*}
\begin{equation*} \label{cklesp}
\eckl[k-1]=\eckp + \eckc.
\end{equation*}
\end{prop}

\begin{proof}
\begin{itemize}
\item Consider a $k$-pm-multichain. Then, the element $e$ following this minimum in the chain has a pointed part whose pointed element $p$ is pointed in the minimum of the chain, and other parts pointed or not. Forgetting the minimum in the chain and splitting parts of the partition $e$, we obtain a $(k-1)$-pm-multichain, whose minimum is pointed in $p$ and a (eventually empty) set of $(k-1)$-multichains, whose minimum has only one part, pointed or not.

\item Consider a $k$-um-multichain. Then, the element $e$ following this minimum in the chain has at least a non-pointed part. Forgetting the minimum in the chain and splitting parts of the partition $e$, we obtain a non-empty set of $(k-1)$-um-multichains and a set of $(k-1)$-pm-multichains.

\item Consider a $k$-multichain. The relation is obtained by splitting parts in the minimum of the chain.

\item Consider a $k$-multichain with multiplicity $((a_1, \ldots, a_k), \m)$ of semi-pointed partitions in a maximal interval $[\m, \pmax[V][V_1]]$ of the unbounded poset. We obtain a $k+1$-multichain by adding at the beginning of the chain the partition $\m$ in one part bounding the interval. The unique part of the partition $\m$ can be bounded or not. There is then a one-to-one correspondence between the pairs $((a_1, \ldots, a_k), \m)$ and the chains $(\m, a_1, \ldots, a_k)$. The chain $(\m, a_1, \ldots, a_k)$ is a $(k+1)$-multichain whose minimum has only one part.
\end{itemize}
\end{proof}

\subsection{Dimension of the homology of the semi-pointed partition poset}

We now want to compute the dimension of the unique homology group of semi-pointed partition posets and of their maximal intervals. We use the relations between species of Proposition \ref{relesp}: these relations imply relations between generating series which are written in the following proposition. These relations will give us closed explicit formulae for the dimension of the homologies. These relations are true for all positive integer $k$. However, the coefficient of series $\ckp$, $\ckc$, $\ckg$ and $\ckl$ are polynomial in $k$: the relations are then true for all integers.

\begin{prop}\label{2.6} For all integers $k$, the generating series $\ckp$, $\ckc$, $\ckg$ and $\ckl$ satisfy the following relations: 
\begin{equation}\label{2.1}
\ckp=\ckp[k-1] \times e^{\ckp[k-1] + \ckc[k-1]},
\end{equation}
\begin{equation}\label{2.2}
\ckc= e^{\ckp[k-1]} \left( e^{\ckc[k-1]} -1 \right),
\end{equation}
\begin{equation}\label{2.3}
\ckg= \exp \left( \ckp[k] + \ckc[k]\right) -1,
\end{equation}
\begin{equation} \label{ckl}
\ckl[k-1]=\ckp + \ckc.
\end{equation}
\end{prop}

Computing the first terms, we obtain:

\begin{prop} The series $\ckp[-1]$ and $\ckg[-1]$ are linked by: 
\begin{equation} \label{eqx} 
\left\lbrace \begin{array}{c}
x=\ckp[-1] \left(1+ \ckg[-1] \right), (a)\\
y = \ckg[-1] +1 -e^{\ckp[-1]}. (b)
\end{array} \right.
\end{equation}
\end{prop}

\begin{proof}
We first compute $\ckp[1]$ and $\ckc[1]$. Let us recall that the coefficient of $\frac{x^\p y^\np}{\p!\np!}$ in $\ckp[1]$ (resp. $\ckc[1]$) corresponds to the number of $1$-pm-multichains (resp. $1$-um-multichains), on a set of "pointable" elements of size $\p$ and a set of "non-pointable" 
elements of size $\np$. There are $\p$ such elements in the pointed case of $\ckp[1]$ and $1$ if $\np$ is positive, $0$ otherwise, in the non-pointed case $\ckc[1]$. We thus obtain: 
\begin{equation*}
\ckp[1]=x e^{x+y} \text{ and } \ckc[1] =e^{x+y}-e^{x}.
\end{equation*}
Using relations between generating series \eqref{2.1} and \eqref{2.2}, we obtain:
\begin{equation} \label{syst}
\left\lbrace \begin{array}{c}
x e^{x+y} = \ckp[0] e^{\ckp[0]+\ckc[0]} \\
e^{x+y}-e^x = e^{\ckp[0]+\ckc[0]}-e^{\ckp[0]}.
\end{array} \right.
\end{equation}
 These equations imply the following functional equation which uniquely determines $\ckp[0]$: 
 \begin{equation*}
 x e^{x+y} = \ckp[0] \left( e^{x+y}-e^x + e^{\ckp[0]}\right).
 \end{equation*}
 Solving this equation and reporting it in \ref{syst}, we obtain:
\begin{equation*}
\ckp[0]=x \text{ and } \ckc[0]=y.
\end{equation*}

These values together with \eqref{2.1} and \eqref{2.2} for $k=0$, with $\ckc[-1]$ eliminated thanks to \eqref{2.3}, give the relations of the proposition.

\end{proof}

\begin{rque}
The coefficient of $\frac{x^\p y^{\np}}{\p! \np!}$ in the series $\ckg[1]$  gives the number of semi-pointed partitions in $\grd$. This series satisfies: 
\begin{equation*}
\ckg[1]= \exp\left((x+1)e^{x+y}-e^{x}\right)-1.
\end{equation*}

We summarize in Table \ref{tablepartsp} the first values for this series.

\begin{table}
\begin{center}
\begin{tabular}{|c|c|c|c|c|c|c|c|}
\cline{1-7}
$ \p$ | $\np$ &{0} & {1} &  {2} &  {3} &  {4} &  {5}  \\ \cline{1-7} \cline{1-7}
 {0} &0 & 1 & 2 & 5 & 15 & 52    \\ \cline{1-7}
 {1} & 1 & 3 & 8 & 25 & 89 & 354\\ \cline{1-7}
 {2} & 3 & 10 & 35 & 133 & 552 & 2493 \\ \cline{1-7}
 {3} & 10 & 41 & 173 & 768 & 3637 \\ \cline{1-6}
 {4} & 41 & 196 & 953 & 4815 \\ \cline{1-5}
 {5} & 196 & 1057 & 5785 \\ \cline{1-4}
 {6} & 1057 & 6322 \\ \cline{1-3}
 {7} & 6322 \\\cline{1-2}
\end{tabular}
\caption{Number of semi-pointed partitions in $\grd$ \label{tablepartsp}}
\end{center}
\end{table}

The first line corresponds to the case of partitions of a set of size $n$: this cardinality is given by the $n^\text{th}$ Bell number $B_n$. The first column $\np=0$ corresponds to the case of pointed partitions which are enumerated by $\sum_{\p=0}^{n} \binom{n}{\p} (n-\p)^\p$.
\end{rque}

Equations \eqref{eqx} give the following equation for $\ckp[-1]$: 
\begin{equation} \label{eqcmup}
x= \ckp[-1] \left( y + e^{\ckp[-1]} \right).
\end{equation}

We now use Lagrange's inversion theorem, which can be found in \cite[Theorem 5.4.2, Corollary 5.4.3]{Stanley2001} for instance, to obtain the result. 

\begin{thm}\label{serieg} Let $V_1$ be a set of cardinality $\p$ and $V_2$ be a set of cardinality $\np$.
The sum of dimensions of the unique homology group of every maximal interval, whose minimum has a unique part which is pointed, in the semi-pointed partition poset over $V_1$ and $V_2$ is given by:
\begin{equation*}
\frac{(\p+\np-1)!}{(\p-1)!} (\p+\np)^{\p-1}.
\end{equation*}

The dimension of the unique homology group of the augmented poset of semi-pointed partitions over $V_1$ and $V_2$ is given by:
\begin{equation*}
(-1)^{\p+\np-1} \frac{(\p+\np-1)!}{(\p-1)!} (\p+\np-1)^{\p-1}.
\end{equation*}

The sum of dimensions of the unique homology group of every maximal interval in the semi-pointed partition poset over $V_1$ and $V_2$ is given by: 
\begin{equation*}
(-1)^{\p+\np-1} \frac{(\p+\np-1)!}{\p!} (\p+\np)^{\p}.
\end{equation*} 
\end{thm}

\begin{proof}
\begin{itemize}
\item Let us now compute the coefficients of $\ckp[-1]$. We apply to \eqref{eqcmup} Lagrange's inversion theorem with $F(x) = x \left( y + e^x \right)$. We get: 
\begin{equation*}
\ckp[-1] = \sum_{\p \geq 1} \frac{x^\p}{\p!} \left( \frac{\partial}{\partial z}\right)^{\p-1} \left( \frac{1}{y+e^z}\right)^{\p}_{z=0},
\end{equation*} 
where $f(z)_{z=0}$ stands for the value in $z=0$ of the map $f(z)$.
The computation of the derivative, using Binomial theorem, gives:
\begin{equation*}
\left( \frac{\partial}{\partial z}\right)^{\p-1} \left( \frac{1}{y+e^z}\right)^{\p} = \sum_{\np \geq 0} \binom{-\p}{\np} y^\np \left( \frac{\partial}{\partial z}\right)^{\p-1}  e^{z(-\p-\np)}.
\end{equation*}
We hence obtain the following equation: 
\begin{equation*}
\left( \frac{\partial}{\partial z}\right)^{\p-1} \left( \frac{1}{y+e^z}\right)^{\p}_{z=0} =  \sum_{\np \geq 0} (-1)^{\p+\np-1} \frac{(\p+\np-1)!}{(\p-1)!} \frac{y^\np}{\np!}  \left(\p + \np \right)^{\p-1}.
\end{equation*}
The result immediately follows from this equation.

\item The second equation follows from Lagrange's inversion theorem with $H(x) = y-1 + e^x$ using \eqref{eqx}(b). The coefficient $\frac{x^\p y^l}{\p! l!}$ of the series we want to compute is then given by the coefficient of $\frac{x^{\p-1}y^l}{(\p-1)! l!}$ in: 
\begin{equation*}
H'(x) \left( \frac{x}{F(x)}\right)^\p = e^x (y + e^x)^{-\p}.
\end{equation*} 

\item The third equation is obtained thanks to the differential equation of the following lemma: 
\begin{lem}\label{lem2.15}
The generating series $\ckl[-2]$ and $\ckp[-1]$ satisfy the following differential equation: 
\begin{equation*}
x \frac{\partial \ckl[-2]}{\partial x} = x \frac{\partial \ckp[-1]}{\partial x} + y \frac{\partial \ckp[-1]}{\partial y}. 
\end{equation*}
\end{lem}

Writing $\ckl[-2] = \sum_{\p \geq 0, \np \geq 0, (\p,\np) \neq (0,0)} a_{\p,\np} \frac{x^\p y^\np}{\p! \np!}$, the above differential equation of Lemma \ref{lem2.15} implies the following relation, for all $\p \geq 0, \np \geq 0, (\p,\np) \neq (0,0)$: 
\begin{equation*}
\p a_{\p,\np} = - \frac{(\p+\np-1)!}{(\p-1)!} (\p+\np)^{\p-1} (\p+\np).
\end{equation*}
The result immediately follows.
\end{itemize}
\end{proof}

We now prove Lemma \ref{lem2.15} used in the proof above.
\begin{proof}[Proof of Lemma \ref{lem2.15}]
Using \eqref{ckl}, we get: 
\begin{equation*}
x \frac{\partial \ckl[-2]}{\partial x} = x \frac{\partial \ckp[-1]}{\partial x} + x \frac{\partial \ckc[-1]}{\partial x}.
\end{equation*}

The lemma is then true if and only if the following equality holds:
\begin{equation*}
y \frac{\partial \ckp[-1]}{\partial y} =  x \frac{\partial \ckc[-1]}{\partial x}.
\end{equation*}

To obtain this equation, we differentiate Equation \eqref{eqcmup} with respect to $x$ and $y$. We obtain:
\begin{equation*}
y \dycmup=\frac{-y \ckp[-1]}{y + e^{\ckp[-1]}+ \ckp[-1] e^{\ckp[-1]}}
\end{equation*}
and
\begin{equation*}
\dxcmup =\frac{1}{y + e^{\ckp[-1]}+ \ckp[-1] e^{\ckp[-1]}}.
\end{equation*}

We also differentiate Equation \eqref{eqx}(b) with respect to $x$. We get: 
\begin{equation*}
\dxcmuc =\frac{\dxcmup \left( e^{\ckp[-1]} - e^{\ckp[-1] + \ckc[-1]} \right) }{ e^{\ckp[-1] + \ckc[-1]}}.
\end{equation*}

However, according to Equation \eqref{eqx}(b), we have $e^{\ckp[-1]} - e^{\ckp[-1] + \ckc[-1]} = -y$ and according to Equation \eqref{eqx}(a), we have $e^{-\ckp[-1] - \ckc[-1]} = \frac{\ckp[-1]}{x}$, hence the result.

\end{proof}

\subsection{Action of the symmetric groups on the homology of the semi-pointed partition poset}

The permutation of the elements of $V_1 \cup V_2$ induced by an element of $\mathfrak{S}_{V_1} \times \mathfrak{S}_{V_2}$ implies a permutation of semi-pointed partitions with the same number of parts. This permutation preserves the order and then induces an action on the homology of the semi-pointed partition poset on $V_1$ and $V_2$. The equations on species of Proposition \ref{relesp} enable us to compute this action.

\begin{thm} Let $\sigma$ be a permutation of $\mathfrak{S}_{V_1} \times \mathfrak{S}_{V_2}$ with $\lambda_i$ $i$-cycles on $V_1$ and $\mu_j$ $j$-cycles on $V_2$.
The character for the action of the symmetric groups on the sum of homologies of the maximal intervals whose least elements are pointed is given on $\sigma$ by:
\begin{multline}\label{ciscp}
\left( (-1)^{\lambda_1+\mu_1-1} \frac{(\lambda_1+\mu_1-1)!}{(\lambda_1-1)!} (\lambda_1+\mu_1)^{\lambda_1-1} \right) \\ \times \prod_{l \geq 2}(-1)^{\mu_l + \lambda_l-1} \frac{(\lambda_l + \mu_l-1)!}{ \lambda_l!} {l^{\lambda_l}\alpha_l^{\lambda_l-1}} \left(-\alpha_l \lambda_l + (\lambda_l + \mu_l) l \lambda_l \right),
\end{multline} 
where $\alpha_m = \sum_{n | m} n (\lambda_n + \mu_n)$.

The character which evaluated in $t=0$ gives the action of the symmetric groups on the sum of homologies of the maximal intervals, and which evaluated in $t=1$ gives the action of the symmetric groups on the augmented semi-pointed partition posets, is given on $\sigma$ by:
\begin{multline} \label{cist}
\frac{1}{t} \prod_m (-1)^{\lambda_{m}+\mu_{m}-1}m^{\mu_{m}} \mu_m! \left(\alpha_m - g_{m,t} \right)^{\lambda_m-1} \left[m \lambda_m \binom{\lambda_{m}+\mu_{m}-g_{m,t}}{\mu_m} \right.\\ \left. - \left(\alpha_m - g_{m,t} \right) \binom{\lambda_{m}+\mu_{m}-g_{m,t}-1}{\mu_m} \right]
\end{multline} 
where the sum runs over the integers $m$ such that $\lambda_m + \mu_m >0$, $\alpha_m = \sum_{n | m} n (\lambda_n + \mu_n)$, $\mu$ is the usual number theory Möbius function and $g_{m,t} = \frac{1}{m} \sum_{k|m} \mu(k)t^{m/k}$.

The evaluation in $t=0$ of the preceeding character, which corresponds to the action of the symmetric groups on the sum of homologies of the maximal intervals, is given on $\sigma$ by:
\begin{multline} \label{ciscl}
\left((-1)^{\lambda_{n}+\mu_{n}-1}n^{\mu_{n}}\alpha_{n}^{\lambda_{n}-1} \frac{\mu({n})(\lambda_{n}+\mu_{n}-1)!}{n \times \lambda_{n}!} \left[\lambda_{n}(1-{n}) + \alpha_{n} \right] \right)\times \\
\prod_{\substack{m>n \\ \lambda_m + \mu_m >0}} (-1)^{\lambda_{m}+\mu_{m}-1}m^{\mu_{m}} \alpha_{m}^{\lambda_{m}-1} \frac{(\lambda_{m}+\mu_{m}-1)!}{\lambda_{m}!} \lambda_m \left[ m (\lambda_m + \mu_m) - \alpha_m \right] 
\end{multline} 
where $\alpha_m = \sum_{n | m} n (\lambda_n + \mu_n)$, $\mu$ is the usual number theory Möbius function and $n = \min\{m|\alpha_m=m(\lambda_m + \mu_m)\}$.
\end{thm}

\begin{proof}\begin{enumerate}
\item According to the equations of Proposition \ref{relesp}, the cycle index series of the species $\emup$ satisfies:
\begin{equation*}
p_1 = \zmup (q_1 + \mathbb{E} \circ \zmup),
\end{equation*}
where the cycle index series of $\mathbb{E}$ is $\exp(\sum_{k\geq 1} \frac{p_k}{k})$.

We want to compute the coefficient of $\prod_{i,j} p_i^{\lambda_i} q_j^{\mu_j}$ in $\zmup$, hence the residue:
\begin{equation*}
I  = \int \zmup \prod_{i,j \geq 1} \frac{dp_i}{p_i^{\lambda_i+1}}\frac{dq_j}{q_j^{\mu_j+1}}.
\end{equation*}

We use the change of variables $z_l = p_l \circ \zmup$. Then $z_l$ satisfy:
\begin{equation*}
p_l = z_l \left( q_l+ \exp\left( \sum_{k \geq 1 } \frac{z_{kl}}{k}\right)\right)
\end{equation*}
and we obtain:
\begin{equation*}
I = \int z_1 \prod_{l \geq 1, j} \left(\int \frac{q_l+ \exp\left( \sum_{k \geq 1 } \frac{z_{kl}}{k}\right) + z_l \exp\left( \sum_{k \geq 1 } \frac{z_{kl}}{k}\right)}{\left( q_l+ \exp\left( \sum_{k \geq 1 } \frac{z_{kl}}{k}\right) \right)^{\lambda_l+1}} \frac{dq_l}{q_l^{\mu_l+1}} \right) \frac{dz_j}{z_j^{\lambda_j+1}} 
\end{equation*} 

We compute 
\begin{multline*}
\int \frac{q_l+ \exp\left( \sum_{k \geq 1 } \frac{z_{kl}}{k}\right) + z_l \exp\left( \sum_{k \geq 1 } \frac{z_{kl}}{k}\right)}{\left( q_l+ \exp\left( \sum_{k \geq 1 } \frac{z_{kl}}{k}\right) \right)^{\lambda_l+1}} \frac{dq_l}{q_l^{\mu_l+1}} \\
= \exp\left( \sum_{k \geq 1 } \frac{z_{kl}}{k}(-\lambda_l-\mu_l)\right),
\end{multline*}

Hence, the result \eqref{ciscp}.

\item To obtain this result, we have to compute the coefficient of $\prod_{i,j} p_i^{\lambda_i} q_j^{\mu_j}$ in
\begin{equation*}
\frac{\mathbb{E} \circ t \zkl[-2]-1}{t}. 
\end{equation*}
We recall the notations of Example \ref{esp} : the cycle index series of operad $\comm$ is equal to the cycle index series $\mathbb{E}-1$ and its inverse for the substitution is operad $\lie$.
Using relations between species established in Proposition \ref{relesp}, we obtain the following equality, stated in term of cycle index series.
\begin{equation*}
\frac{\comm \circ t \zkl[-2] }{t} = \comm \circ t \Sigma \lie \circ \left(q_1 + \comm \circ \zmup \right).
\end{equation*}

The substitution $\comm \circ t \Sigma \lie$ is computed in \cite{Getzler}. We then have to compute the following residue
\begin{equation*}
\frac{1}{t}\int \prod_l \left( q_l + \exp\left( \sum_k \frac{p_{kl}}{k} \circ \zmup\right)\right)^{g_{l,t}} \prod_{i,j} \frac{dp_i dq_j}{p_i^{\lambda_i+1} q_j^{\mu_j+1}},
\end{equation*}
where $g_{l,t} = \frac{1}{l}\sum_{k|l} \mu(k) t^{l/k}$, with $\mu$ the usual Möbius function of an integer.

We use the change of variable $z_l = p_l \circ \zmup$ of the precedent point. We then use Newton binomial series and integrate with respect to the variables $q_l$.

Then, denoting $\alpha_n = \sum_{k|n} k (\lambda_k + \mu_k)$, we obtain
\begin{equation*}
\prod_{k,l} \exp\left( (-\lambda_l-\mu_l) \frac{z_{kl}}{k}\right)=  \prod_{m} \exp \left( - \alpha_m \frac{z_m}{m}\right)
\end{equation*}
and \begin{equation*}
\prod_{k,l} \exp\left( g_{l,t} \frac{z_{kl}}{k}\right)=  \prod_{m} \exp \left( t^m \frac{z_m}{m}\right).
\end{equation*}
We use these equalities to reorganize the terms in the integral and use the development of the exponential to integrate with respect to the variables $z_m$.

\item We evaluate the preceding result in $t=0$. Therefore, let us remark that when $n$ is chosen such that $n(\lambda_n + \mu_n) = \alpha_n$, the term associated with $n$ in the product \eqref{cist} is :
\begin{equation*}
(-1)^{\lambda_{n}+\mu_{n}-1}n^{\mu_{n}}  \left(\alpha_n - g_{n,t} \right)^{\lambda_n-1} g_{n,t} (\lambda_n(1-n)+\alpha_n)\prod_{k=1}^{\mu_{n}-1} (\lambda_{n}+k-g_{n,t}) 
\end{equation*}
This term is divisible by $t$ and the limit of $\frac{g_{n,t}}{t}$ when $t$ tends to $0$ is $\frac{\mu(n)}{n}$. We thus obtain the result.  
\end{enumerate}
\end{proof}

\begin{rque}
When $\lambda_m=0$ for all $m$, we obtain the character for the action of the symmetric group on the homology of partition posets, computed by R. Stanley \cite{Stanpart} and P. Hanlon \cite{Hanpart}, which was linked with the operad $\lie$ by \cite{Joypart} and \cite{Fressepart} (see also \cite{Fressepart} for more bibliographical details).

When $\mu_m=0$ for all $m$, we obtain the result of F. Chapoton and B. Vallette \cite{ChVal} on the homology of pointed partition posets : the character for the action of the symmetric group is linked with the operad $\prelie$. 
\end{rque}

\section{Incidence Hopf algebra} 

We apply in this section the construction detailed in \cite{IHA} of W. Schmitt of an incidence Hopf algebra associated to a family of posets satisfying some axioms. We then compute the coproduct in this Hopf algebra to identify the studied Hopf algebra with an Hopf algebra of generating series, which will enable us to compute some characteristic polynomials.

\subsection{Generalities on incidence Hopf algebra}

	All the definitions recalled here are extracted from the article of W. Schmitt \cite{IHA}.  

	 A family of intervals $\mathcal{P}$ is \emph{interval closed}, if it is non-empty and, for all $P \in \mathcal{P}$ and $x \leq y \in P$, the interval $[x,y]$ belongs to $\mathcal{P}$. An \emph{order compatible relation} on an interval closed family $\mathcal{P}$ is an equivalence relation $\sim$ such that $P \sim Q$ if and only if there exists a bijection $\phi: P \rightarrow Q$ such that $[0_P,x] \sim [0_Q,\phi(x)]$ and $[x, 1_P] \sim [\phi(x), 1_Q]$, for all $x \in  P$. Isomorphism of posets is an example of order compatible relation. 
	
	Given $\sim$ an order compatible relation on an interval closed family $\mathcal{P}$, we consider the quotient set $\mathcal{P}/\sim$ and denote by $[P]$ the $\sim$-equivalence class of a poset $P \in \mathcal{P}$. We define a $\mathbb{Q}$-coalgebra $\CP$ as follow: 
	
\begin{prop} \cite[Theorem 3.1]{IHA}
Let $\CP$ denote the free $\mathbb{Q}$-module generated by $\mathcal{P}/\sim$. We define linear maps $\Delta: \CP \rightarrow \CP \otimes \CP$ and $\epsilon: \CP \rightarrow \mathbb{Q}$ by: 
\begin{equation*}
\Delta[P]=\sum_{x \in P} [0_P,x] \otimes [x,1_P]
\end{equation*}
and
\begin{equation*}
\epsilon[P]=\delta_{|P|,1},
\end{equation*}
where $\delta_{i,j}$ is the Kronecker symbol. 
Then, $\CP$ is a coalgebra with comultiplication $\Delta$ and counit $\epsilon$. 
\end{prop}

	The \emph{direct product} of posets $P_1$ and $P_2$ is the cartesian product $P_1 \times P_2$ partially ordered by the relation $(x_1,x_2) \leq (y_1,y_2)$ if and only if $x_i \leq y_i$ in $P_i$, for $i=1,2$. A \emph{hereditary family} is an interval closed family which is also closed under formation of direct products. Let $\sim$ be an order compatible relation on $\mathcal{P}$ which is also a semigroup congruence, i.e., whenever $P \sim Q$ in $\mathcal{P}$, then $P \times R \sim Q \times R$ and $R \times P \sim R \times Q$, for all $R \in \mathcal{P}$. This relation is \emph{reduced} if whenever $|R|=1$, then $P \times R \sim R \times P \sim P$ : all trivial intervals are then equivalent and yields to a unit element for the product on the quotient. These hypotheses ensure that the product will be well defined on the quotient. The obtained unit is denoted by $\nu$. An order compatible relation on a hereditary family $\mathcal{P}$ which is also a reduced congruence is called a \emph{Hopf relation} on $\mathcal{P}$. Isomorphism of posets is a Hopf relation for instance.
	
	\begin{prop}[\cite{AIC}] Let $\sim$ be a Hopf relation on a hereditary family $\mathcal{P}$. Then $\HP=(\CP, \times, \Delta, \nu, \epsilon,S)$ is a Hopf algebra over $\mathbb{Q}$. 
	\end{prop}

\subsection{Description}

We now consider maximal intervals in the dual $\grd^*$ of the semi-pointed partition poset $\grd$. Using the notations introduced in the first section, we will write $\pinko$ for $\pinko[n][ \p][ 0]$ or $\pinko[n][ \p][ 1]$. We denote by $\pmax$ the least element of $\pinko$, whose parts are of size $1$ and $\pinkmax$ the greatest element in only one part.

The following proposition ensures that the family $\mathcal{F}$ of direct product of maximal intervals in a semi-pointed partition poset is a hereditary family:

\begin{prop}[Intervals in semi-pointed partition posets] \label{Intspp}
Let $p$ be a semi-pointed partition in the poset $\pinko$. The interval $[ p;\pinkmax]$ in $\pinko$ is isomorphic to a poset of semi-pointed partitions $\pinko[j][l][o]$, where $j$ is the number of parts in $p$ and $l$ is the number of pointed parts in $p$.

The interval $[\pmax ; p]$ is isomorphic to a product of semi-pointed partitions poset with a factor $\pinko[n_j][ \p_j][ 1]$ for every pointed part of $p$ of size $n_j$ with $\p_j$ elements pointed in $\pmax$ and a factor $\pinko[n_j][ \p_j][ 0]$ for every non-pointed part of $p$ of size $n_j$ with $\p_j$ elements pointed in $\pmax$.
\end{prop}

We consider the relation $\equiv$ given by $P \equiv Q$ if $P$ and $Q$ are isomorphic and if the number of pointed and non pointed parts are the same in the greatest and least elements of $P$ and $Q$. This relation is a Hopf relation. Considering the hereditary family $\mathcal{F}$ and the Hopf relation $\equiv$, we can apply the construction of W. Schmitt presented in \cite{IHA}, and recalled in the above section, to obtain an incidence Hopf algebra $\mathcal{I}$. This family is an algebra over the set of maximal intervals in semi-pointed partition posets.

\subsection{Computation of the coproduct}

Let us describe now more precisely the coproduct in $\mathcal{I}$. Using the decomposition of intervals described in Proposition \ref{Intspp}, we obtain the following description of the coproduct:
\begin{equation*}
\Delta(\pinko)= \sum_{j=1}^n \sum_{\substack{n_1, \ldots, n_j \geq 1, \\ \sum_{i=1}^j n_i = n}} \sum_{\substack{\p_1, \ldots, \p_j \geq 0 \\ \sum_{i=1}^j \p_i = \p}}
\sum_{\substack{o_1, \ldots, o_j \in \{0;1\} \\ o_i \leq \p_i,\\ o \leq \sum_{i=1}^j o_i \leq j-1+o }} \cnlo \prod_{i=1}^{j} \pinko[n_i][ \p_i][ o_i] \otimes \pinko[j][\sum_{i=1}^j o_i][ o]
\end{equation*}
where $\cnlo$ is the number of partitions having $j$ parts, of size $n_1, \ldots, n_j$, with
 $\p_1, \ldots, \p_j$ elements in each part pointed in $\pmax$ and with the $i$th part pointed if $o_i$ is $1$ and non pointed otherwise.
 
Counting the numbers of partitions in $\cnlo$ gives the following theorem:
 
 \begin{thm}
 The coproduct in the incidence Hopf algebra $\mathcal{I}$ of semi-pointed partition poset is given by:
 \begin{multline}
\Delta\left(\frac{\pinko[k+l][k]}{l!(k-o)!}\right)= \sum_{p+q \geq 1} \sum_{(l_i,k_i)} \prod_{i=1}^p \frac{\pinko[l_i+k_i][ k_i][ 1]}{l_i!(k_i-1)!} \\
 \prod_{i=p+1}^{p+q} \frac{\pinko[l_i+k_i][ k_i][ 0]}{l_i!k_i!} \otimes \frac{\pinko[p+q][p][ o]}{q! (p-o)!},
\end{multline}
where the second sum runs over the $p+q$-tuples $(l_1, \ldots, l_{p+q})$ and $(k_1, \ldots, k_{p+q})$ satisfying $l_1, \ldots, l_p \geq 0$, $l_{p+1}, \ldots, l_{p+q} \geq 1$, $\sum_{i=1}^{p+q} l_i = l$, $k_1, \ldots, k_p \geq 1$, $k_{p+1}, \ldots, k_{p+q} \geq 0$ and $\sum_{i=1}^{p+q} k_i = k$.
 \end{thm}

 \begin{proof} In the non pointed case ($o=0$), the coefficient $\cnlo$ is given by: 
  \begin{equation*}
 \frac{k!l! \prod_{i=1}^p k_i}{\prod_{i=1}^{p+q} k_i! l_i! p! q!}.
 \end{equation*}
 Indeed, we make $p+q$ packets of elements and the first $p$ packets have to be pointed.
For the pointed case ($o=1$), to ensure that the greatest part is pointed, for instance in $1$, we fix that the first packet is pointed in $1$. The coefficient $\cnlo$ is then given by:
  \begin{equation*}
 \frac{(k-1)!l! \prod_{i=2}^p k_i}{(k_1-1)!\prod_{i=2}^{p+q} k_i! l_i! p! q!}.
 \end{equation*}
 \end{proof}

\subsection{Computation of characteristic polynomials of maximal intervals in the poset}

We now give an interpretation of the computation of the coproduct in the previous subsection. This interpretation will help us computing characteristic polynomials.

\begin{prop}
The incidence Hopf algebra $\mathcal{I}$ of the family of semi-pointed partition posets is isomorphic to the Hopf algebra structure on the polynomial algebra in the variables $(a_{k,l}^o)_{k,l \geq 1, o \in \{0,1\}}$ given by the composition of pairs $(F,G)$ of formal series of the following form: 
\begin{align*}
F&=x + \sum_{l,k \geq 1} a_{k,l}^0 \frac{x^k}{k!} \frac{y^l}{l!} \\
G&=y + \sum_{l,k \geq 1} k a_{k,l}^1 \frac{x^k}{k!} \frac{y^l}{l!}.
\end{align*}
\end{prop}

As a corollary, the Möbius numbers of the intervals $\pinko[n][\p][0]$ and $\pinko[n][\p][1]$ are respectively the coefficients of $A$ and $B$, where $(A,B)$ satisfy:
\begin{align*}
(e^B-1)e^A&=x \\
Ae^{A+B}&=y.
\end{align*}

By comparison with equations of Proposition \ref{2.6}, we obtain an other proof that $A=\ckp[-1]$ and $B=\ckc[-1]$.

We recall the definition of the characteristic polynomial of an interval $P$ whose minimum is denoted by $\hat{0}$ and whose maximum is denoted by $\hat{1}$: 
\begin{equation*}
\sum_{x \in P} \mu(\hat{0}, x) t^{rk(\hat{1})-rk(x)},
\end{equation*}
where $rk(x)$ is the rank of $x$, i.e. the length of the longest chain between $\hat{0}$ and $x$.

The characteristic polynomials $\chi^\bullet$ and $\chi^\times$ of $\pinko[n][\p][1]$ and $\pinko[n][\p][0]$ are then given by:

\begin{align*}
\chi^\bullet&=\frac{(e^{t\ckc[-1]}-1)e^{t\ckp[-1]}}{t} \\
\chi^\times&={\ckp[-1]} e^{t(\ckp[-1]+\ckc[-1])}
\end{align*}

The above relations implies the following proposition:
\begin{prop} The characteristic polynomial of the interval $\pinko[n][\p][1]$ is given by:
\begin{equation}
(t-1) \left(t-\p-\np \right)^{\p-2} \prod_{i=\p+1}^{\p+\np} \left(t - i \right).
\end{equation}

The characteristic polynomial of the poset $\grd$ is given by:
\begin{equation}
 \left(t-\p-\np \right)^{\p-1} \prod_{i=\p+1}^{\p+\np} \left(t - i \right).
\end{equation}

The characteristic polynomial of the poset $\hgrd$ is given by:
\begin{equation}
t \times \left(t-\p-\np \right)^{\p-1} \prod_{i=\p+1}^{\p+\np} \left(t - i \right) - \left(1-\p-\np \right)^{\p-1} \prod_{i=\p+1}^{\p+\np} \left(1 - i \right).
\end{equation}
\end{prop}

\begin{rque}
The obtained characteristic polynomial is linked with the cycle index series \ref{cist} computed previously.
\end{rque}

\begin{proof}
\begin{itemize}
\item We use the following relation between $\ckp[-1]$ and $\ckc[-1]$, obtained from Proposition \ref{2.6}:
\begin{equation}
e^{\ckp[-1]+\ckc[-1]}=y+e^{\ckp[-1]}.
\end{equation}
Then, the coefficient of $x^l y^p$ in $\chi^\bullet$ is given by the following residue: 
\begin{equation*}
I=\iint \ckp[-1] (y+e^{\ckp[-1]})^t \frac{dx}{x^{\p+1}} \frac{dy}{y^{\np+1}}.
\end{equation*}
We use the substitution $z=\ckp[-1]$ to obtain: 
\begin{align*}
I&=\iint \frac{z (y+e^{z})^t}{(y+e^{z})^{\p+1}} \left( (y+e^z) + ze^z\right) \frac{dz}{z^{\p+1}} \frac{dy}{y^{\np+1}}\\
I&=\iint z (y+e^{z})^{t-\p} \frac{dz}{z^{\np+1}} \frac{dy}{y^{\np+1}} + \iint z^2 (y+e^{z})^{t-\p-1} e^z \frac{dz}{z^{\p+1}} \frac{dy}{y^{\np+1}}.
\end{align*}

We expand $(y+e^{z})^{t-\p}$ and $(y+e^{z})^{t-\p-1}$ and take the coefficient of $y^\np$. The integral $I$ is then given by: 
\begin{equation*}
I=\int \binom{t-\p}{\np} e^{(t-\p-\np)z} \frac{dz}{z^{\p}}+\int \binom{t-\p-1}{\np} e^{(t-\np-\p)z} \frac{dz}{z^{\p-1}}.
\end{equation*}
 As $\int e^{az} \frac{dz}{z^n} = \frac{a^{n-1}}{(n-1)!}$, this gives the result.
 
 \item The generating series $\mathcal{S}$ of such characteristic polynomials can be seen as a generating series in $t$ whose coefficients are generating series in $x$ and $y$. Then, the coefficient of $t^{\np-1}$ is the sum of Möbius numbers of partitions of $\grd[n][\p]$ in $p$ parts, weighted by $\frac{x^{\p}y^{\np}}{\p! \np!}$. We thus obtain the relation:
 \begin{equation*}
 \mathcal{S}=\frac{e^{t\left(\ckp[-1]+ \ckc[-1]\right)}-1}{t}.
\end{equation*}   
 The result is obtained by applying Lagrange inversion formula to this relation.
 
 \item The characteristic polynomial of $\hgrd$ is divisible by $(t-1)$, as the poset is bounded, and only differs from $t \mathcal{S}$ by its constant term.
 
\end{itemize}
\end{proof}

\bibliographystyle{alpha}
   \bibliography{bibli}
\end{document}